\documentclass[graybox]{svmult}

\usepackage{type1cm}        
\usepackage{makeidx}         
\usepackage{graphicx}        
\usepackage{multicol}        
\usepackage[bottom]{footmisc}
\usepackage{amssymb,mathtools,amsfonts,amsbsy,amsopn,amsxtra,amscd,amstext}
 
\makeindex             

\begin{document}

\title*{On slice regular  Bergman spaces  and fiber bundles}
\author{Jos\'e Oscar Gonz\'alez-Cervantes}
\institute{Jos\'e Oscar Gonz\'alez-Cervantes \at Departamento de Matem\'aticas, ESFM-Instituto Polit\'ecnico Nacional. 07338, Ciudad M\'exico, M\'exico, \email{jogc200678@gmail.com}}
%
%
\maketitle

\abstract*{The present work shows that  the quaternionic slice regular  Bergman space is the base space of 
 a coordinate sphere bundle.}

 \abstract{
Recently, fiber bundle theory has been widely used in the study of the slice regular functions and continuing with this line of research, the present work shows that  the quaternionic slice regular  Bergman space is the base space of 
 a coordinate sphere bundle and   some properties of   quaternionic slice regular  
Bergman theory   are studied from the  point of view of  the  fiber bundle theory.}

\section{Introduction}
\label{sec:1}
Fiber bundle theory has been commonly used to interpret some physical  
 and mathematical  phenomena,   see    \cite{BP,BD,Bredon,NS,W}. 
  Bergman's theory  has  been developed in hypercomplex analysis    
	see  \cite{bradel,bds,1SRBergman, 2SRBergman,CGS3,const,conskrau,conskrau2,delanghe,shavas,shavas2,shavas3}. 
	In particular,   works  \cite{1SRBergman,2SRBergman,CGS3} 
were the first to introduce the theory of slice regular  Bergman spaces.   
 In this paper  the fiber bundle theory is used to  study  the second notion  of quaternionic Bergman space introduced in 
 \cite{1SRBergman} that consists of a family of slice regular Bergman    associated to a slice generate by $\{1, {\bf i}\}$, 
where  ${\bf i} \in \mathbb S^2$.
For ease of reading, Section \ref{2} presents basic concepts of the quaternionic slice regular Bergman space  and fiber bundle theory and Section \ref{S3} presents the main results.

\section{Preliminaries}\label{2}

\subsection{Quaternionic slice regular Bergman spaces}

The concepts and  results of  this subsection are given in \cite{newadvances,1SRBergman,2SRBergman,CGS3,CSS4,GenSS,advances}.

The skew-field  of quaternions   $\mathbb H$ consists of  $q=x_0  + x_{1} {e_1} +x_{2} e_2 + x_{3} e_3$ where $x_0, x_1, x_2, x_3\in \mathbb R$    and   
the quaternionic units satisfy $e_1^2=e_2^2=e_3^2=-1$,  $e_1e_2=-e_2e_1=e_3$, $e_2e_3=-e_3e_2=e_1$, $e_3e_1=-e_1e_3=e_2$. The set 
 $\{e_1,e_2,e_3\}$ is  called   the standard basis of $\mathbb R^3$.
The vector part of $q$ is denoted by    ${\bf{q}}= x_{1} {e_1} +x_{2} e_2 + x_{3} e_3$ and is usually  identified with the vector  ${\bf{q}}= (x_{1} ,x_{2}   ,x_{3})\in \mathbb R^3$.  The  real part is $q_0=x_0$. The  quaternionic conjugate of $q$  is   $\bar q=q_0-{\bf q} $ and  its   norm  is  $\|q\|:=\sqrt{x_0^2 +x_1^2+x_2^3+x_3^2}= \sqrt{q\bar q} = \sqrt{\bar q  q}$.
 Denote   $\mathbb B^4:=   \{q \in \mathbb H\ \mid \ \|q\|<1 \}$,  
$\mathbb{S}^2:=  \{{\bf q}\in\mathbb R^3  \mid \|{\bf q}\| =1  \}$, $\mathbb{S}^3:=\{ {  q}\in\mathbb H \mid  \|{  q}\| =1   \}$ and 
	  $ T:=\{ ({\bf i},{\bf j}) \in \mathbb S^2 \times \mathbb S^2   \  
	  \mid    \{{\bf i},{\bf j},{\bf ij}\} \textrm{ and }
	  \{e_1,e_2,e_3\} \textrm{ are   co-oriented}\}$. 
		
	Note that    ${\bf i}^{2}=-1$  for any ${\bf i}\in \mathbb S^2$, so 
   $\mathbb{C}({\bf i}):=\{x+{\bf i}y; \ |\ x,y\in\mathbb{R}\}\approx \mathbb C$ as  fields and from the   algebra of quaternions we see that  $\mathbb H=\bigcup_{{\bf i}\in \mathbb S^2} \mathbb C({\bf i})$.
 The  quaternionic rotations that preserve $\mathbb R^3$ are given by  $q \mapsto v q \bar v$ for all $q\in \mathbb H$, where  
 $v\in \mathbb S^3$, see \cite{HJ}. Then let us consider  $R_{v}({\bf i}, {\bf j}):= (v{\bf i}\bar v, v{\bf j}\bar v)$ for all $({\bf i}, {\bf j})\in T$.

\begin{definition}
 Let $\Omega\subset\mathbb H$ be an open set. A real differentiable function $f:\Omega\to \mathbb{H}$  is called  slice regular function on $\Omega$,   
if   
\begin{align*}
\overline{\partial}_{{\bf i}}f\mid_{_{\Omega\cap \mathbb C({\bf i})}}:=\frac{1}{2}\left (\frac{\partial}{\partial x}+{\bf i} \frac{\partial}{\partial y}\right )f\mid_{_{\Omega\cap \mathbb C({\bf i})}}=0  \textrm{ on  $\Omega\cap \mathbb C({\bf i})$,}
\end{align*}
for all ${\bf i}\in \mathbb{S}^2$ and the Cullen derivative of $f$  is   
$f'=\displaystyle 
 {\partial}_{{\bf i}}f\mid_{_{\Omega\cap \mathbb C({\bf i})}} 
= \frac{\partial}{\partial x} f\mid_{_{\Omega\cap \mathbb C({\bf i})}}= \partial_xf\mid_{_{\Omega\cap \mathbb C({\bf i})}} $. 
 The quaternionic right module of the  slice regular functions on $\Omega$ is denoted by $\mathcal{SR}(\Omega)$.
\end{definition}

\begin{definition}
 A set $U\subset\mathbb H$  is called axially symmetric if  $U\cap \mathbb R\neq \emptyset$ and if $x+{\bf i}y \in U$ with $x,y\in\mathbb R$   then $\{x+{\bf j}y \ \mid  \ {\bf j}\in\mathbb{S}^2\}\subset U$. In addition, a   domain  $U\subset\mathbb H$ is a slice domain, or s-domain, if   $U_{\bf i} := U\cap \mathbb C({\bf i})$ is  a domain in   $\mathbb C({\bf i})$  for all ${\bf i}\in\mathbb S^2$.  Given an axially symmetric s-domain  $\Omega\subset \mathbb H$   denote 
 $$S_{\Omega}:=\{  (x,y)\in \mathbb R^2 \ \mid \ \textrm{ there exists } {\bf i}\in \mathbb S^2 \textrm{ and } x+{\bf i}y\in \Omega  \}$$
\end{definition}

The following theorem presents two   important results in the slice regular function theory.

\begin{theorem}\label{1.3} Let   $\Omega\subset \mathbb H$  be an axially symmetric s-domain.  
Given $f\in\mathcal{SR}(\Omega)$ and $({\bf i},{\bf j})\in T$.    
\begin{enumerate}
\item (Splitting Lemma) There exist $f_1 ,f_2\in Hol(\Omega_{\bf i})$   such that
  $f_{\mid_{\Omega_{\bf i}}} =f_1  + f_2  {\bf j}$ on $\Omega_{\bf i}$.
  
\item (Representation Formula) For every  $q=x+{\bf I}_q y \in \Omega$ where $x,y\in\mathbb R$ and  ${\bf I}_q \in \mathbb S^2$  we have 
\begin{align*}
f(x+{\bf I}_q y) = \frac {1}{2}[   f(x+{\bf i}y)+ f(x-{\bf i}y)]
+ \frac {1}{2} {\bf I}_q {\bf i}[ f(x-{\bf i}y)- f(x+{\bf i}y)].
\end{align*} 
\end{enumerate}
\end{theorem}

The above results help us directly introduce the following operators and their properties.

\begin{definition}Let   $\Omega\subset \mathbb H$  be an axially symmetric s-domain and $({\bf i},{\bf j})\in T$. Define 
$     Q_{ {\bf i},{\bf j}} :    \mathcal{SR}(\Omega) \to \textrm{Hol}(\Omega_{  {\bf i}})+ \textrm{Hol}(\Omega_{ {\bf i}}){  {\bf j} }$ and 
$ P_{  {\bf i},{\bf j} } :  \textrm{Hol}(\Omega_{  {\bf i}})+ \textrm{Hol}(\Omega_{  {\bf i}}){  {\bf j}} \to  \mathcal{SR}(\Omega)$ 
 as follows:
$ Q_{ {\bf i} , {\bf j} } [f ]= f\mid_{\Omega_{{\bf i}}} = f_1+f_2{\bf j}$ for all $f\in\mathcal{SR}(\Omega)$ where $f_1,f_2\in \textrm{Hol}(\Omega_{\bf i})$,  and 
 $   P_{ {\bf i},{\bf j} }[g](q)= \frac{1}{2}\left[(1+ {\bf I}_q{ \bf  i})g(x-y{{\bf i}}) + (1- {\bf I}_q {  {\bf i}}) g(x+y {  {\bf i} })\right]
	$,  for all  $g\in  \textrm{Hol}(\Omega_{  {\bf i} })+ \textrm{Hol}(\Omega_{  {\bf i} }){  {\bf j} }$ where  $ q=x+{\bf I}_qy \in \Omega$, respectively.  In addition,   $ P_{  {\bf i} ,  {\bf j}  }\circ Q_{ {\bf i} ,{\bf j}  }= \mathcal I_{\mathcal{SR}(\Omega)}$
 and    $Q_{  {\bf i} ,  {\bf j} }\circ P_{ {\bf i} ,  {\bf j}   }= \mathcal I_{ \textrm{Hol}(\Omega_{ {\bf i} })+ \textrm{Hol}(\Omega_{ {\bf i} }){ {\bf j} }} $,
 where $\mathcal I_{\mathcal{SR}(\Omega)} $ and $ \mathcal I_{ \textrm{Hol}(\Omega_{ {\bf i} })+ \textrm{Hol}(\Omega_{ {\bf i} }){ {\bf j} }}$ are identity operators. 
	
On the other hand, given  $f\in \mathcal{SR}(\Omega)$ and     $({\bf i},{\bf j}) \in T$, from direct computation we see that   
\begin{align*} &  D_1[f, {\bf i}, {\bf j}   ] :=   \frac{    Q_{{\bf i},{\bf j}}[f] + \overline{   Q_{{\bf i},{\bf j}}[f]  }   }{ 2 }  ,  & \  D_2[f, {\bf i}, {\bf j}     ] :=  - \frac{  Q_{{\bf i},{\bf j} }[f] {\bf i}  +  \overline{   Q_{{\bf i} , {\bf j} }[f]   {\bf i}  }   }{2}   , \nonumber \\
&  D_3[f, {\bf i}, {\bf j}      ] :=   - \frac{ Q_{{\bf i},{\bf j} }[f]  {\bf j}   + \overline{   Q_{{\bf i},{\bf j} }[f]   {\bf j} }   }{2}   , &  \  D_4[f, {\bf i}, {\bf j}     ] :=    \frac{ Q_{ {\bf i}, {\bf j} }[f] {\bf j}{\bf i}  + \overline{   Q_{ {\bf i},{\bf j} }[f]   {\bf j}{\bf i}    }   }{2} 
\end{align*}
are the real components of  $Q_{{\bf i},{\bf j}}[f]$ in terms of    base $\{1, {\bf i}, {\bf j}, {\bf ij}\}$.  Therefore, 
$ Q_{{\bf i},{\bf j}}[g] =  D_1[g, {\bf i}, {\bf j}   ]  +  D_2[g, {\bf i}, {\bf j}   ] {\bf i}   +  D_3[g, {\bf i}, {\bf j}   ] {\bf j}+  D_4[g, {\bf i}, {\bf j}  ] {\bf i}{\bf j} $.
\end{definition}

\begin{definition}\label{def123} 

Given     $f,g \in \mathcal {SR}(\Omega) $  and $({\bf i}, {\bf j}) \in T$ define   
 the $\bullet_{{\bf i}, {\bf j}}-$product: $ 
f\bullet_{_{{\bf i}, {\bf j}}} g := P_{\bf{i},{\bf j}} [ \  f_1g_1+ f_2g_2 {\bf j} \ ]$,
   where $f_1,f_2,g_1,g_2\in \textrm{Hol}(\Omega_{\bf i})$ are such that  $Q_{{\bf i}, {\bf j}}[f]=f_1+f_2{\bf j}$ and  $Q_{{\bf i}, {\bf j}}[g]=g_1+g_2{\bf j}$.

In particular, if $\Omega = \mathbb B^4(0,1)$ then  there exists two sequences of quaterinons $(a_n)_{n\geq 0}$ and $(b_n)_{n\geq 0}$ such that 
 $
f(q)= \sum_{n=0}^{\infty} q^n a_n$,  $  g(q)= \sum_{n=0}^{\infty} q^n b_n$  then their   $*-$product is  
$ 
f*g(q):= \sum_{n=0}^{\infty} q^n  \sum_{k=0}^{n} a_k b_{n-k}$ 
 for all $q\in\mathbb B^4(0,1)$.\end{definition}

\begin{definition}

 Let $\Omega\subseteq  \mathbb{H}$ be an axially symmetric s-domain.
We introduce a family  of quaternionic slice regular  Bergman spaces. Given ${\bf{i}}\in \mathbb{S}^2$   denote 
  $$
\displaystyle{     \mathcal A_{ \bf i  }  (\Omega ):=    \left\{    f\in \mathcal{SR}(\Omega) \   \mid  \
    \displaystyle \|f     \|^2_{{\mathcal A}_{\bf i} (\Omega)}:= \int_{\Omega_{\bf i}}|f_{\mid_{\Omega_{{\bf i}}}}|^2 d\sigma_{\bf i} <\infty   \right\} }
  $$
equipped with the  norm
$
  \displaystyle\|f\|_{{\mathcal A}_{\bf i}(\Omega )}:= \left(\int_{\Omega_{\bf i}}|f_{\mid_{\Omega_{{\bf i}}}}|^2 d\sigma_{\bf i}\right)^{\frac{1}{2}}, \quad \forall f\in \mathcal A_{\bf i}(\Omega),
$ and with  the scalar product
 $
\langle f,g\rangle_{\mathcal A_{\bf i}   (\Omega)}=\int_{\Omega_{{\bf i}}} \overline{f\mid_{\Omega_{{\bf i}} }}\, g \mid_{\Omega_{{\bf i}} }\,d\sigma_{{\bf i}}, 
 \quad \forall f,g \in \mathcal A_{\bf i}(\Omega)$,
 where $d \sigma_{\bf i}$ denotes the  element of area in the complex plane $\mathbb C({\bf i})$.

\end{definition}

\begin{remark}\label{Rem11} Given   ${\bf i}, {\bf k}\in \mathbb S^2$ and $f \in \mathcal A_{\bf k}   (\Omega)$.   Then Representation Formula 
implies that $
    |f(x+y{\bf{i}})|^2\leq 2 \left[ | f(x+y{\bf k})|^2 + | f(x-y{\bf k})|^2\right]$ and    the integration on  $S_\Omega$     gives us  
    $  \int_{\Omega_{\bf i}}|f_{\Omega_{\bf i}}|^2d\sigma_{\bf i} \leq 2 \|f\|_{{\mathcal A_{\bf k}}(\Omega)}^2 $, 
     i.e., 
  $\|f\|_{\mathcal A_{{\bf i}}(\Omega )} \leq \sqrt{2} \|f\|_{\mathcal A_{\bf k}(\Omega)}$. Therefore,  $    \mathcal A_{\bf i}(\Omega)$ and 
  $   \mathcal A_{\bf k}(\Omega )$ are  equal functions sets and   have equivalent norms. 
\end{remark}

The following theorem  was proved in \cite{1SRBergman}.
\begin{theorem} 
 Let $\Omega$ be a bounded axially symmetric s-domain and ${\bf i}\in \mathbb S^2$. Then  $\left( \mathcal A_{\bf i} (\Omega), \langle \cdot, \cdot\rangle_{\mathcal A_{\bf i}(\Omega)} \right)$ is a   quaternionic right   Hilbert module. In addition, given $q\in \Omega$ let    $\phi_q$  
the  evaluation functional at point $q$   on    $\mathcal A _{\bf i} (\Omega)$. 
Then 
there  exists an unique  function $K_q({\bf i}) \in \mathcal A _{\bf i} (\Omega) $ such that
   $
  \phi_q[f ]=\langle K_q({\bf i}), f \rangle_{\mathcal A_{\bf i} (\Omega ) }
   $  for all 
   $f  \in \mathcal A_{\bf i} (\Omega )$,   
   due to    
the evaluation functional on $ \mathcal A_{\bf i}(\Omega)$ is 
   a bounded quaternionic right-linear functional and T. Riesz representation theorem for quaternionic right  Hilbert modules. 		
		  \end{theorem}

\begin{remark}
If we denote  $\mathcal K_{\Omega_{\bf i}}(q,\cdot):= \bar K_q({\bf i})$ then   the slice regular Bergman kernel  associated to $\Omega_{\bf i}$ 
is given by    $\mathcal K_{\Omega_{\bf i}}(\cdot,\cdot):\Omega_{\bf i}\times \Omega_{\bf i} \to \mathbb H$.   Representation Formula shows  that  
 $
\mathcal K_{\Omega}: \Omega\times \Omega \to \mathbb{H}$ satisfies: 
$$
\mathcal K_{\Omega}(x+y{\bf{k}}, r):=\frac{1}{2}(1-{\bf k} {\bf i}) \mathcal K_{\Omega_{\bf i}}(x+y{\bf i},r)+ \frac{1}{2}(1+{\bf k} {\bf i})\mathcal K_{\Omega_{\bf i}}(x+y{\bf i},r),
 $$
for all  $ {\bf k}\in \mathbb S^2$ and $r\in \Omega_{\bf i}$.
\end{remark}

In this work we will focus on the slice regular Bergman space associated with the quaternionic unit ball $\mathbb B^4$ to simplify the notation and computations.   If  $\Omega = \mathbb B^4$ then 
$\int_{\mathbb B^4_{\bf i}}|f_{\mid\mathbb B^4_{\bf i}}|^2d\sigma_{\bf i} =    \int_{ D}|f(x+{\bf i}y)|^2 d\mu_{x,y}$ and  
$\langle f, g \rangle _{ \mathcal A_{\bf i}( \mathbb B^4 ) } =    \int_{D} \overline{ f(x+{\bf i}y)}
g (x+{\bf i}y)    d\mu_{x,y}$,  for all $f, g \in \mathcal A_{\bf i}( \mathbb B^4)$,  
where $D:=\{(x,y)\in \mathbb R^2 \ \mid \  x^2+y^2< 1 \} $ and  $d\mu_{x,y}$ is the area differential in $\mathbb R^2$.

\begin{definition}Define   
$L :=\{ f:\mathbb B^4 \to \mathbb H \ \mid \ \int_{D}\|f(x+{\bf i}y)\|^2 d\mu_{x,y}<\infty, \quad \forall {\bf i}\in \mathbb S^2\}$. 
 Given     ${\bf i}\in \mathbb S^2$. The slice regular Bergman projection  $B_{\bf i}: L\to \mathcal A_{\bf i}(\mathbb B^4 ) $ 
 is       $B_{\bf i}[f](q) =\int_{\mathbb B^4_{{\bf i}}} \mathcal K_{\mathbb B^4}(q,z)
 g(z)d\sigma_{{\bf i}}$  for all $q\in \mathbb B^4$,  
where $ \mathcal K_{\mathbb B^4}$ is the slice regular  Bergman Kernel of $\mathcal A_{\bf i}(\mathbb B^4 )$.

And, in addition, we introduce the Toeplitz operators in $\mathcal A_{\bf i}(\mathbb B^4 )$ extending    the concept of Toeplitz operator in the  complex Bergman theory  given  in    \cite{ACM}.  Consider  $\alpha\in C(\overline{\mathbb B^4_{\bf i}},\mathbb H)$ and   as the quaternionic product is non conmutative we have two kinds of Toepliz  operators associated to the slice regular Bergman space $\mathcal A_{{\bf i}}(\mathbb B^4 )$.  The left-Toeplitz operator, or only Toeplitz operator, with symbol $\alpha$ is given by  $T^{\bf i}_{\alpha} f =B_{\bf i}[  \alpha  f  ]$ for all $f  \in L$, i.e., 
$$(T^{\bf i}_{\alpha} f ) (q)= \int_{\mathbb B^4_{{\bf i}}} \mathcal K_{\mathbb B^4}(q,z)
  \alpha(z)  f(z)   d\sigma_{{\bf i}}, 
 \quad \forall q  \in \mathbb B^4  $$    and the
 the right-Toeplitz operator  with symbol $\alpha$ is given by $T^{\bf i}_{r,\alpha} f     =B_{\bf i}[  f \alpha]$ for all $f  \in L$.   
\end{definition}

\subsection{On fiber bundle theory}\label{Subse2.2}

The following basic concepts about fiber bundle theory are given in   \cite{Bredon}.

\begin{definition} \label{def1598} A fiber bundle 
    $(X, {\bf{P}}, B, F)$ consists of 
   Hausdorff spaces    $X$,  $B$,   $F$
        and    a topological action group  $K$ on $F$ as a group of homeomorphisms. The continuous surjective map  
 ${\bf{P}}: X\to B$ is the  bundle projection and each   element of $ B$ has  a neighborhood $U\subset B$    and 
   a homeomorphism  called trivialization $\varphi : U\times F\to {\bf{P}}^{-1}(U)$    such that ${\bf{P}}\circ \varphi (x, y) = x$ for all $x\in U$ and   $y\in F$. 
	In addition,  the family of  all   triviali\-za\-tions     $\Phi$ satisfies the following: 
	 \begin{enumerate} 
\item 	If $\varphi:U\times F\to {\bf P}^{-1}(U)$ belongs to  $\Phi$ and    $V\subset U$ 
	then     $\varphi_{\mid_{V\times F}}$ belongs to $\Phi$. 
	\item  
 If $\varphi,  \phi\in\Phi$   over $U\subset B$  then there exists a map $\psi :U\to K$ such that 
$ \phi (x,y) = \varphi ( x, \psi(u) (y) )$ 
for all $x\in U$ and $y\in F$. 
\item   $\Phi$ is a maximal family.
\end{enumerate}
A continuous map  $S: B\to X$ is a  section of $(X, {\bf{P}}, B, F)$,  if ${\bf P}\circ S(x)=x$ for all $x\in B$. An 
 arbitrary map $\mathfrak K:A\to B$  generates a  fiber bundle $(\mathfrak K^*(X), {\bf P}', A, F)$   called  the induced bundle from ${\bf{P}}$ by $\mathfrak K$, or pullback  bundle,  where 
 $\mathfrak K^*(X) := \{  (a,x)\in A\times X \ \mid \ {\bf P}(x)= \mathfrak K(a)     \} $   and    $ {\bf P}'(a,x)= a$ for all  $(a,x)\in \mathfrak K^*(X)$. 
If $F$ and $K$ are  the   sphere and  
 a subgroup of  the orthogonal group of an Euclidean  space respectively, then 
$(X, {\bf{P}}, B, F)$  is called     sphere bundle.
 In addition,  $(X, {\bf{P}}, B, F)$  is called a coordinate bundle if 
     the  mappings  $ \varphi_{v,x}:  F \to {\bf P}^{-1} (x) $  given by  $\varphi_{v,x}(y) =   \varphi_{v}(x,y )$ where  $\varphi_{v}$  is a trivialization and  $(x,y)\in U\times F$  define the  transition functions   $g_{v,w}(x)=\varphi_{v,x}^{-1}\circ \varphi_{w,x}$ that  satisfy:    
$g_{v,w}g_{w,l}  = g_{v,l} $,     
$g_{v,v}  = e $ and  $[g_{v,w} ]^{-1} = g_{w,v}$,
 where $e$ is the identity element.

 Let 	$ (X_1, {\bf P}_1, B_1, F_1)$ and  $  (X_2, {\bf P}_2, B_2, F_2)$ be two fiber bundles then   a  morphism  
  $\Gamma : (X_1, {\bf P}_1, B_1, F_1) \to   (X_2, {\bf P}_2, B_2, F_2)$    is a
pair of  continuous maps  $\Gamma_1: X_1 \to X_2$ and $  \Gamma_2 : B_1 \to  B_2$  such that ${\bf P}_2\circ \Gamma_1 = \Gamma_2 \circ {\bf P_1}:X_1 \to B_2$.
 In addition,  $\Gamma$ is a isomorphism if there exists a 
   morphism   $\Gamma^{-1} : (X_2, {\bf P}_2, B_2, F_2) \to   (X_1, {\bf P}_1, B_1, F_1)$  such  that  
$\Gamma\circ \Gamma^{-1}$ and $\Gamma^{-1}\circ \Gamma$ are the identity morphisms. 
  \end{definition}

\section{Main results}\label{S3}

 We will define the Hausdorff spaces and the continuous mappings required to  explain  the quaternionic  slice regular Bergman space  in terms of   the  fiber   bundle theory.

\begin{definition}\label{def1} The set   $HL^2( D)$  is formed by  pairs of conjugated harmonic functions  $(a,b)$  on $D$ such that 
		$$  \|a\|_{L^{2}}^2:= \displaystyle \int_{ D } |a(x,y)|^2 d\mu_{x,y}= \int_{ D } |a|^2 d\mu  , \ \ \    \|b\|_{L^{2}}^2:=\displaystyle \int_{D} |b|^2 d\mu < \infty.$$
Set  
   $	\mathcal {HL}( D  ):= \left\{  \left( \left( \begin{array}{cc}   a  &  b   \\  c & d  \end{array}\right)  ,   ({\bf k},{\bf l  })  \right) \   \mid   \    (a,b), (c,d) \in   HL^2(D) , \  ({\bf k},{\bf l }) \in T \right\}
 $  
equipped with the following metric
   \begin{align*} 
    & d  \left(  \   \left( \left( \begin{array}{cc}   a  &  b   \\  c & d  \end{array}\right)  ,  ({\bf k},{\bf l })  \right)  , \left( \left( \begin{array}{cc}   r  & s    \\ t & u  \end{array}\right)  ,  ({\bf m},{\bf n })  \right)  \    \right) \\ 
     := &    \|a-r\|_{L^{2} }   + \|b-s\|_{L^{2}  }  +\|c-t\|_{L^{2} }    + \|d-u\|_{L^{2} }   +   \| ({\bf k},{\bf l })- ({\bf m},{\bf n })\|_{\mathbb  R^6},
\end{align*}
for all 
$ \left( \left( \begin{array}{cc}   a  &  b   \\  c & d  \end{array}\right)  ,  ({\bf k},{\bf l })  \right)  , \left( \left( \begin{array}{cc}   r  &  s   \\  t & u  \end{array}\right)  ,  ({\bf m},{\bf n })  \right) \in  \mathcal {HL}( D) $.

 \end{definition}

\begin{definition}\label{def122} Given ${\bf i}\in \mathbb S^2$  consider the    metric induced by 
 $\|\cdot\|_{\mathcal{A}_{\bf i}  (\mathbb B^4) } $  on 
$\mathcal{A}_{\bf i}  (\mathbb B^4 ) $. In addition,     $\mathcal{A}_{\bf i}  (\mathbb B^4 )   \times T$ is equipped with the following metric:  
$\rho_{ {\bf i} } \left( f, ({\bf k}, {\bf l }) ),  (g, ({\bf m}, {\bf n }) ) \right) =
\|f-g\|_{\mathcal{A}_{\bf i}  (\mathbb B^4 ) } + \|({\bf k}, {\bf l })-({\bf m}, {\bf n }) \|_{\mathbb R^6}$, 
for all    $(f, ({\bf k}, {\bf l }) ),  (g, ({\bf m}, {\bf n }) )  \in \mathcal{A}_{\bf i}  (\mathbb B^4 )   \times T$.

 \end{definition}

\begin{remark} Given $ \left( \left( \begin{array}{cc}   a  &  b   \\  c & d  \end{array}\right)  ,  ({\bf k},{\bf l })  \right) 
 \in  \mathcal {HL}(D) $.  From   Representation Theorem,  Splitting Lemma and Remark \ref{Rem11}  we have that  
 $P_{{\bf k},{\bf l }}[  a +b {\bf k} +c{\bf l }  + d {\bf k}{\bf l }       ]  \in \mathcal{SR}(\mathbb B^4)$  and 
$
  \| P_{{\bf k},{\bf l }}[  a +b {\bf k} +c{\bf l }  + d {\bf k}{\bf l }       ]\|_{ \mathcal{A}_{\bf i}(\mathbb B^4)}^2    
\leq   
2 \left( 
   \|a\|^2_{L^{2}}   + \|b\|^2_{L^{2}}    +\|c\|^2_{L^{2}}   + \|d\|^2_{L^{2}} 
 \right)$. 
\end{remark}

\begin{definition} Given ${\bf i}\in \mathbb S^2$ and using   the previous remark    
 the operator    	${ \mathcal{P}}_{\mathbb B^4} : \mathcal {HL}(D)   \to \mathcal{A}_{\bf i}(\mathbb B^4 )$  given by 
$
 { \mathcal{P}}_{\mathbb B^4 }   \left( \left( \begin{array}{cc}   a  &  b   \\  c & d  \end{array}\right)  ,  ({\bf k},{\bf l })  \right)    := P_{{\bf k},{\bf l }}[  a +b {\bf k} +c{\bf l }  + d {\bf k}{\bf l }       ]
 $ 
 is well defined. Given $ ({\bf k},{\bf l }) \in T$ denote  
$$\displaystyle 
{\bf S}_{{\bf k}, {\bf l }}[f]:= \left( \left( \begin{array}{cc}  \displaystyle   D_1[f, {\bf k}, {\bf l } ]  &  D_2[f,  {\bf k}, {\bf l }] \\   D_3[f, {\bf k}, {\bf l }  ]  &  D_4[f, {\bf k}, {\bf l }  ]    \end{array}\right)  ,   ( {\bf k}, {\bf l }  )  \right)       
,$$   
for  all $f\in  \mathcal{A}_{\bf i}(\mathbb B^4)$. In addition,     let  $U\subset \mathcal{A}_{\bf i}(\mathbb B^4)$ be a neighborhood and   $u\in \mathbb S^3$  
define       
  $ \varphi_u  [ f, ({\bf k}, {\bf l }) ] :=  \left( \left( \begin{array}{cc}  \displaystyle   D_1[f, u{\bf k}\bar u, u{\bf l }\bar u  ]  &  D_2[f, u{\bf k}\bar u, u{\bf l }\bar u  ] \\   D_3[f, u{\bf k}\bar u, u{\bf l }\bar u  ]  &  D_4[f, u{\bf k}\bar u, u{\bf l }\bar u  ]    \end{array}\right)  ,   (u{\bf k}\bar u, u{\bf l }\bar u )  \right)  $,    
for all $(f, ({\bf k}, {\bf l }))\in U\times T$.  
 
\end{definition}

\begin{proposition}Given 
$\left( \left( \begin{array}{cc}   a  &  b   \\  c & d  \end{array}\right)  ,  ({\bf i},{\bf j })  \right)  , \left( \left( \begin{array}{cc}   r  &  s   \\  t & u  \end{array}\right)  ,  ({\bf k},{\bf l })  \right) \in \mathcal {HL}( D)  $
 we have that  
\begin{align*} &  \|    P_{{\bf i}, {\bf j}}[  a +b {\bf i} +c{\bf j}  + d {\bf i}{\bf j}       ] -   P_{{\bf k},{\bf l}}[  r +  s {\bf k} + t {\bf l}  + u {\bf kl}      ] \|_{   \mathcal{A}_{\bf i}(\mathbb B^4  ) }   \\
\leq & 
8   \left( 1 +  \|        r \|_{  L^{2} }
+  \|        s \|_{  L^{2} }
 +\|        t  \|_{  L^{2} }
+\|        u \|_{  L^{2} } \right)  
 d  \left(  \   \left( \left( \begin{array}{cc}   a  &  b   \\  c & d  \end{array}\right)  ,  ({\bf i},{\bf j })  \right)  , \left( \left( \begin{array}{cc}   r  &  s   \\  t & u  \end{array}\right)  ,  ({\bf k},{\bf l })  \right)  \    \right). 
\end{align*}
\end{proposition}

\begin{proof} From    Representation Theorem and  Splitting Lemma   we have that 
\begin{align*} &  \|    P_{{\bf i}, {\bf j}}[  a +b {\bf i} +c{\bf j}  + d {\bf i}{\bf j}       ] -   P_{{\bf k},{\bf l}}[  r +  s {\bf k} + t {\bf l}  + u {\bf kl}      ] \|_{   \mathcal{A}_{\bf i}(\mathbb B^4  ) }   \\
\leq &  \|    \frac{1}{2} (1+{\bf i}{\bf i})\left( a +b  {\bf i} +c  {\bf j}  + d  {\bf i}{\bf j} \right)\circ \zeta    - \frac{1}{2}(1+{\bf i}{\bf k}) 
\left(   r +  s   {\bf k} + t  {\bf l}  + u  {\bf kl}     \right)\circ \zeta  \|_{  L^{2}( \mathbb B^4_{\bf i } , \mathbb H) } \\
+ & \|    \frac{1}{2} (1-{\bf i}{\bf i})\left( a  +b {\bf i} +c {\bf j}  + d  {\bf i}{\bf j} \right)    - \frac{1}{2}(1-{\bf i}{\bf k}) 
\left(   r  +  s  {\bf k} + t  {\bf l}  + u  {\bf kl}     \right)  \|_{  L^{2}( \mathbb B^4_{\bf i } , \mathbb H) } \\
\leq & \frac{1}{2} \left\{ \|    a  -r    \|_{  L^{2} }  
+ \| b-s   \|_{  L^{2}  } 
  + \|  c-t  \|_{  L^{2}  }  
 + \| d -u    \|_{  L^{2} } \right\}\\
& +\frac{1}{2} \left\{  \|        s  \|_{  L^{2} }  \  \|{\bf i}- {\bf k}\|_{\mathbb H}  +
+ \|t  \|_{  L^{2} }  \  \|{\bf j}- {\bf l}  \|_{\mathbb H} +    \|u   \|_{  L^{2}  }  \  \| 
{\bf i}{\bf j} - {\bf kl}\|_{\mathbb H}
  \right\} \\
& +
  \frac{1}{2} \left\{ \|    
    a- r   \|_{  L^{2} }  +  \|  b- s     \|_{  L^{2}  }  +  \| c- t \|_{  L^{2}  } + 
     \| d- u  \|_{  L^{2}  } \right\}\\
&+
 \frac{1}{2}\left\{ \|    
   r  \|_{  L^{2} } \  \|{\bf  i} -{\bf k}\|_{\mathbb H} +\|  s \|_{  L^{2}  } \    \| {\bf i} {\bf i} -{\bf k} {\bf k}\|_{\mathbb H} 
   + \| t  \|_{  L^{2} }  \ 
	\| {\bf i} {\bf j} - {\bf k} {\bf l}\|_{\mathbb H}  + \| u  \|_{  L^{2} }  \ 
	\|{\bf i} {\bf i}{\bf j}   - {\bf k}{\bf kl} \|_{ \mathbb H}
  \right\}\\
  &
 +  \frac{1}{2} \left\{ \|       a  -  r   \|_{  L^{2} }    
+ \| b-s \|_{  L^{2}  }   
  + \| c - t  \|_{  L^{2}  }    
  + \|d- u   \|_{  L^{2}  } \right\}\\
 &
 +  \frac{1}{2}\left\{ \|       
 s \|_{  L^{2}  } \    \|  {\bf i}-   {\bf k}\|_{\mathbb H} 
 + \| t \|_{  L^{2} }  \  \| {\bf j} 
-  {\bf l}  \|_{\mathbb H} + \| u  \|_{  L^{2}  } \   \|
     {\bf i}{\bf j} -   {\bf kl}   \|_{    \mathbb H  } \right\}\\
&  + \frac{1}{2} \left\{\|    
    r-a \|_{  L^{2} }   
	+ \| s -  b \|_{  L^{2}  }    + 
	\| t - c \|_{  L^{2} }   
	+ \|  u -  d\|_{  L^{2}  }\} \right\}\\
&
+ \frac{1}{2}\left\{ \|    
       r  \|_{  L^{2} }  \  \| {\bf k}- {\bf i}\|_{\mathbb H}
			+  \| s \|_{  L^{2} }  \   
			\| {\bf k} {\bf k}- {\bf i} {\bf i}\|_{\mathbb H}  +
 \|t  \|_{  L^{2}  } \   \|
 {\bf k} {\bf l}- {\bf i}{\bf j}\|_{\mathbb H}  +\|  u  \|_{  L^{2}  } \   
 \|{\bf k}{\bf kl}  -{\bf i} {\bf i}{\bf j}\|_{\mathbb H} \right\}  \\ 
\leq & 
8   \left( 1 +  \|       r \|_{  L^{2} }
+  \|       s  \|_{  L^{2} }
 +\|        t  \|_{  L^{2} }
+\|        u  \|_{  L^{2} } \right)  
 d  \left(  \   \left( \left( \begin{array}{cc}   a  &  b   \\  c & d  \end{array}\right)  ,  ({\bf i},{\bf j })  \right)  , \left( \left( \begin{array}{cc}   r  &  s   \\  t  & u  \end{array}\right)  ,  ({\bf k},{\bf l })  \right)  \    \right),  
\end{align*}
where $\zeta (q)= \bar q$ for all $q\in \mathbb H$ and the   relationships   
 $\|      a  \circ \zeta  \|_{  L^{2}( \mathbb B^4_{\bf i } , \mathbb H) }=
 \|      a    \|_{  L^{2}( \mathbb B^4_{\bf i } , \mathbb H) } =   \|    a  \|_{  L^{2} }  $
  and  $  \|
 {\bf k} {\bf l}- {\bf i}{\bf j}\|_{\mathbb H}   \leq  2 \left( \|
 {\bf k}  - {\bf i} \|_{\mathbb H} + 
  \|
  {\bf l} -  {\bf j}\|_{\mathbb H} \right)  \leq 4 \|({\bf i}, {\bf j})-({\bf k}, {\bf l}) \|_{\mathbb R^6} 
 $ 
were used.

\end{proof}

\begin{corollary}\label{BundleProje}
   	${ \mathcal{P}}_{\mathbb B^4 } : \mathcal {HL}( D )   \to \mathcal{A}_{\bf i}(\mathbb B^4 )$
is a continuous surjective operator. 
\end{corollary}

\begin{proposition}  If 
$  (f,  ({\bf i}, {\bf j }) ),  (g,  ({\bf k}, {\bf l }) )  \in   \mathcal{A}_{\bf i}(\mathbb B^4 )  \times T$ then  
\begin{align*}
& d \left(  \   \left( \left( \begin{array}{cc}  \displaystyle   D_1[f, {\bf i}, {\bf j} ]  &  D_2[f,  {\bf i}, {\bf j }] \\   D_3[f, {\bf i}, {\bf j }  ]  &  D_4[f, {\bf i}, {\bf j}  ]    \end{array}\right)  ,   ( {\bf i}, {\bf j }  )  \right) ,  \left( \left( \begin{array}{cc}  \displaystyle   D_1[g, {\bf k}, {\bf l} ]  &  D_2[g,  {\bf k}, {\bf l }] \\   D_3[g, {\bf k}, {\bf l }  ]  &  D_4[g, {\bf k}, {\bf l}  ]    \end{array}\right)  ,   ( {\bf k}, {\bf l }  )  \right) \ \right) \\
  \leq &
 \   4  \left(  1+ (1+ \sqrt{2}) \| g \|_{   \mathcal{A}_{\bf i}(\mathbb B^4 ) }   \right)  \  \rho_{\bf i}\left(  ( f ,   ({\bf i}, {\bf j}  )  ) , ( g,   ( {\bf k}, {\bf l }) ) \right)    .
\end{align*}

\end{proposition}
\begin{proof}
From Representation Formula,  Splitting Lemma and similar  computations to the previous proof we can conclude   that 
\begin{align*}
& d \left(  \   \left( \left( \begin{array}{cc}  \displaystyle   D_1[f, {\bf i}, {\bf j} ]  &  D_2[f,  {\bf i}, {\bf j }] \\   D_3[f, {\bf i}, {\bf j }  ]  &  D_4[f, {\bf i}, {\bf j}  ]    \end{array}\right)  ,   ( {\bf i}, {\bf j }  )  \right) ,  \left( \left( \begin{array}{cc}  \displaystyle   D_1[g, {\bf k}, {\bf l} ]  &  D_2[g,  {\bf k}, {\bf l }] \\   D_3[g, {\bf k}, {\bf l }  ]  &  D_4[g, {\bf k}, {\bf l}  ]    \end{array}\right)  ,   ( {\bf k}, {\bf l }  )  \right) \ \right) \\
\leq &
 4 \left[ \left( \int_{D}\| f(x+{\bf i}y )  - g(x+{\bf i} y)\|^2d\mu_{x,y} \right)^{\frac{1}{2}} \right. \\
  &  + \left(  \int_{D}\| g(x+{\bf i}y )  - g(x+{\bf k} y)\|^2d\mu_{x,y} \right)^{\frac{1}{2}}   \\
& \left.+ \left( \int_{D}\| g(x+{\bf i} y)\|^2d\mu_{x,y} \right)^{\frac{1}{2}}     \| ({\bf i}, {\bf j}) -  ({\bf k}, {\bf l })  \|_{\mathbb R^6} \right]   
+  \|   ({\bf i}, {\bf j}  ) -  ( {\bf k}, {\bf l })  \|_{\mathbb B^6}
\\
 \leq &
 4  (  1+  \| g \|_{   \mathcal{A}_{\bf i}(D ) }  )  \  \rho_{\bf i}\left(  ( f ,   ({\bf i}, {\bf j}  )  ) , ( g,   ( {\bf k}, {\bf l }) ) \right)  
+4 \left(  \int_{ D }\| g(x+{\bf i}y )  - g(x+{\bf k} y)\|^2d\mu_{x,y} \right)^{\frac{1}{2}}    .
\end{align*} 
Given  $(x,y)\in D$ from Remark \ref{Rem11} and  Representation Formula we get    
\begin{align*}
g(x+{\bf i}y )  - g(x+{\bf k} y) = \frac{1}{2} \left[ ({\bf i}- {\bf k}) {\bf k} g(x+{\bf k}y) +  ({\bf i}- {\bf k}) {\bf k} g(x-{\bf k}y)  \right].
\end{align*}  Therefore,
$\left(  \int_{D}\| g(x+{\bf i}y )  - g(x+{\bf k} y)\|^2d\mu_{x,y} \right)^{\frac{1}{2}}   \leq 
 \sqrt{2} \|  {\bf i}- {\bf k} \|_{\mathbb H}  \| g\|_{  \mathcal{A}_{\bf i}(\mathbb B^4  ) }   .
$ and the two  previous inequalities give us the main result.
\end{proof}

\begin{corollary}\label{SecTrivia}
Given $({\bf i}, {\bf j })\in T$ and  $ u \in \mathbb S^3$ then  $ 
{\bf S}_{{\bf i}, {\bf j }}$ and of $ \varphi_u  $ are continuous operators.
\end{corollary}

\begin{proposition}
Consider  $({\bf i}, {\bf j})\in T$. Then       $(\mathcal {HL}(D),  \mathcal{P}_{\mathbb B^4 } ,\mathcal{A}_{\bf i}(\mathbb B^4 ), T  )$  is a coordinate  sphere  bundle where $K=\{R_v  \mid v \in \mathbb S^3 \}$ is the structure group,    $\Phi =\{ \varphi_v    \mid  v\in\mathbb S^3\}$ is the  family of trivializations  and  the operator ${\bf S}_{{\bf i}, {\bf j}}$ is a section.
\end{proposition}
\begin{proof}
From Corollary \ref{BundleProje}   we see  that     	${ \mathcal{P}}_{\mathbb B^4 } : \mathcal {HL}( D)   \to \mathcal{A}_{\bf i}(\mathbb B^4 )$
is a continuous surjective operator and   Corollary 
\ref{SecTrivia}  shows that  ${\bf S}_{{\bf i}, {\bf j }}$ and  $ \varphi_v  $ are continuous operators.
The set  $K=\{R_v \ \mid \  v\in \mathbb S^3\}$ equipped with the composition is a group  of homeomorphisms on $T$ and 
     $\{ \varphi_v \ \mid \ v\in \mathbb S^3 \}$  is the family of trivializations, see    \cite[Proposition 3.6]{O1} since  our functions 
     are axially symmetric   slice regular functions.
The transition functions are given by  
	\begin{align*}
g_{v,w}(f)  =   \varphi_{v, f}^{-1}\circ \varphi_{w, f}({\bf i}, {\bf j})=  \varphi_{v,f}^{-1}\circ \varphi_{v}( f,  R_{\bar v w}({\bf i}, {\bf j}) ) = R_{\bar v w}({\bf i}, {\bf j}), 
\end{align*}
where $v,w\in \mathbb S^3$, $f\in \mathcal{A}_{\bf i}(\mathbb B^4 )$ and   $({\bf i}, {\bf j})\in T$.  From direct computations we see that 
 $$ { \mathcal{P}}_{\mathbb B^4 }  \circ
{\bf S}_{{\bf i}, {\bf j }} [f] ={ \mathcal{P}}_{\mathbb B^4 }  \left( \left( \begin{array}{cc}  \displaystyle   D_1[f, {\bf i}, {\bf j } ]  &  D_2[f,  {\bf i}, {\bf j }] \\   D_3[f, {\bf i}, {\bf j }  ]  &  D_4[f, {\bf i}, {\bf j}  ]    \end{array}\right)  ,   ( {\bf i}, {\bf j }  )  \right) = f,      
$$    
i.e., ${\bf S}_{{\bf i}, {\bf j }}$ is a section of $(\mathcal {HL}(D),  \mathcal{P}_{\mathbb B^4 } ,\mathcal{A}_{\bf i}(\mathbb B^4 ), T  )$.  
 \end{proof}

\begin{remark}
   The operations,  norm and quaternionic  inner product induced  on $ \mathcal {HL}(D) $ by   
	$\mathcal{A}_{\bf i}(\mathbb B^4 )$ can be used to study the  orthogonality, reproducing property of the Bergman kernel, 
	the behavior of  Toepliz operators among others facts from the point of view fiber bundle theory.
\end{remark}

\begin{proposition}
Given  ${\bf i},{\bf k}\in\mathbb S^2$ then 
   $(\mathcal {HL}(D ),  \mathcal{P}_{\mathbb B^4} ,\mathcal{A}_{\bf i}(\mathbb B^4), T  )$  and 
     $(\mathcal {HL}(D),  \mathcal{P}_{\mathbb B^4} ,\mathcal{A}_{\bf k}(\mathbb B^4), T  )$  are isomorphic sphere bundles.
\end{proposition} 
\begin{proof}
The  isomorphism  is $(\mathcal I, I)$, where   $\mathcal I :  \mathcal {HL}(D) \to  \mathcal {HL}(D)$  and  
  $I : \mathcal{A}_{\bf i}(\mathbb B^4 ) \to \mathcal{A}_{\bf k}(\mathbb B^4 ) $  are the identity mappings.
\end{proof}

\begin{remark}
In a profoundly similar way to the computations  presented in \cite[Proposition 3.11]{O1} given 
 $A,B\in  \mathcal {HL}( D)$ we can define $A+B$,  $A\bullet_{_{{\bf i}, {\bf j}}} B, \mathcal D (A),
 \mathcal R_v(A) $ and $ A\ast B$ to see some algebraic properties of $\mathcal P_{\mathbb B^4 }$ such as  $
	 \mathcal P_{\mathbb B^4 } (A+B) =    \mathcal P_{\mathbb B^4 } (A) + \mathcal P_{\mathbb B^4 } (B)$, 
	$	\mathcal P_{\mathbb B^4 } (A\bullet_{_{{\bf i}, {\bf j}}} B) = \mathcal P_{\mathbb B^4  } (A)\bullet_{_{{\bf i}, {\bf j}}} \mathcal P_{\mathbb B^4   } (B)$, 
	$
 (\mathcal P_{\mathbb B^4   } (A))'=   \mathcal P_{\mathbb B^4   } (\mathcal D (A))$, 
 $\mathcal R_v(A) =   P_{v{\bf i}\bar v, v{\bf j}\bar v} [ \ v Q_{{\bf i}, {\bf j}} [\mathcal P_{\mathbb B^4   } (A) ] \bar v \ ] $ and  
 $\mathcal P_{\mathbb B^4   } (A\ast B) =   \mathcal P_{\mathbb B^4   } (A) \ast \mathcal P_{\mathbb B^4   } (B) $.

On the other hand,     the slice regular Bergman projection  $B_{\bf i}: L\to \mathcal A_{\bf i}(\mathbb B^4 ) $ 
   induces  the   pullback  bundle  $(B_{\bf i}^*(\mathcal {HL}(D) ), \mathcal{P}_{\mathbb B^4 }', L, T)$, 
  where  the total space 
 $$B_{\bf i}^*(\mathcal {HL}(D ) ) := \{  (f,A)\in L\times \mathcal {HL}(D) \ \mid \  \mathcal{P}_{\mathbb B^4 }(A)= B_{\bf i}(f)     \}  $$
   and  
  $  \mathcal{P}_{\mathbb B^4 }'(f,A)= f$ for all  $(f,A)\in B_{\bf i}^*(\mathcal {HL}(D )) $. 
 Note that     $(f,A)\in L\times \mathcal {HL}(D)$ iff   $A$  under $\mathcal{P}_{\mathbb B^4 }$ is assigned to
  slice regular part of $f$.

 In addition, Toeplitz operator  associated to 
 $\mathcal A_{\bf i}(\mathbb B^4 )$ are deeply related with   pullback bundles.  Given $\alpha\in C(\overline{\mathbb B^4_{\bf i}},\mathbb H)$ we see that   the  left-  and  right-Toeplitz operator  with symbol $\alpha$, 
    induce   the   pullback  bundles   $((T_{\alpha}^{\bf i})^*(\mathcal {HL}(D) ), \mathcal{P}_{\mathbb B^4 }', L, T)$ and
    $(   (T_{r, \alpha}^{\bf i}  ) ^*(\mathcal {HL}(D) ), \mathcal{P}_{\mathbb B^4 }', L, T)$ respectively, where     
   $$(T_{ \alpha}^{\bf i}  ) ^* (\mathcal {HL}(D ) ) := \{  (f,A)\in L\times \mathcal {HL}(D) \ \mid \  \mathcal{P}_{\mathbb B^4 }(A)=T_{\alpha}(f)     \} , $$
   $$(T_{r, \alpha}^{\bf i}  ) ^* (\mathcal {HL}(D ) ) := \{  (f,A)\in L\times \mathcal {HL}(D) \ \mid \  \mathcal{P}_{\mathbb B^4 }(A)= T_{r,\alpha}(f)     \} , $$
    $  \mathcal{P}_{\mathbb B^4 }'(f,A)= f$ for all  $(f,A)\in (T_{ \alpha}^{\bf i}  ) ^* (\mathcal {HL}(D ) ) $ and 
      $  \mathcal{P}_{\mathbb B^4 }'(f,A)= f$ for all  $(f,A)\in (T_{r, \alpha}^{\bf i}  ) ^* (\mathcal {HL}(D ) ) $.
      
\end{remark}

This work was presented in   talk "On slice regular Bergman space and fiber bundle theory" in the 14th ISAAC Congress 2023.


%
%

\begin{thebibliography}{99.}
%
%













\bibitem{ACM}  Axler, S.,  Conway, J.B.,  McDonald, G. \textit{Toeplitz operators on Bergman spaces}. Can. J. Math. XXXIV, 2, 466--483, (1982). 

 
 
\bibitem{BP} Bernstein, H. J. , Philips, A. \textit{Fiber Bundles and Quantum Theory}, Scientific American, { \bf 245} 1, 122--137, (1981).    
  
\bibitem{BD}  Bleecker, D. \textit{Gauge theory and variational principles}, 
Dover Books on physics Dover Books on mathemtics, Courier Corporation, (2005). 

\bibitem{Bredon} Bredon, G. E.  \textit{Topology and Geometry}, Springer Verlang,  (1913).


\bibitem{bradel}  Brackx, F.,  Delanghe,  R. \textit{Hypercomplex function theory and Hilbert modules with reproducing kernel},
Proc. Amer. Math. Soc., { \bf 37}, 545--576,  (1978).

\bibitem{bds} Brackx, F., Delanghe, R. ,  Sommen, F. {\em Clifford Analysis},
Pitman Res. Notes in Math., 76, (1982).

 
 
\bibitem{newadvances}  Colombo, F.,  Gentili, G.,  Sabadini, I.,  Struppa, D.C.
\textit{Extension results for slice regular functions of a quaternionic variable},
Adv. Math., { \bf 222} , 1793-1808, (2009).


\bibitem{1SRBergman}   Colombo, F.,  Gonz\'alez-Cervantes, J.O.,  Luna-Elizarrar\'as, M. E.,  Sabadini, I.,  Shapiro, M. \textit{On two approaches to the Bergman theory for slice regular functions}.  
{  Advances in Hypercomplex Analysis}, Springer-Indam Series {\bf 1}, (2012).
 
 \bibitem{2SRBergman}    Colombo, F.,  Gonz\'alez-Cervantes, J. O.,  Sabadini, I.
\textit{On slice biregular functions and isomorphisms of
Bergman spaces}, Compl. Var. Ell. Equa. { \bf  57}, 825--839,  (2012).


\bibitem{CGS3}   Colombo, F.,  Gonz\'alez-Cervantes, J. O.,  Sabadini, I., 
\textit{The C-property for slice regular functions and applications to the Bergman space}, Compl. Var. Ell. Equa., { \bf  58}, 1355--1372,  (2013).

\bibitem{CSS4}  Colombo, F.,  Sabadini, I.,     Struppa, D. C. \textit{Entire slice regular functions}. Springer Briefs in Mathematics,  Springer, (2016).


\bibitem{const}  Constales, D. \textit{ The Bergman and Szeg\"o kernels for separately monogenic functions},
Zeit. Anal. Anwen.,  { \bf 9},  97--103, (1990).

\bibitem{conskrau}  Constales, D.,  Krau\ss har, R. S. \textit{Bergman kernels for rectangular domains and multiperiodic functions in Clifford analysis},  Math. Meth. Appl. Sci., {\bf  25}, 1509--1526,  (2002).

  \bibitem{conskrau2}  Constales, D., Krau\ss har, R. S.  \textit{Bergman spaces of higher-dimensiononal hyperbolic polyhedron-type domains I},  Math. Meth. Appl. Sci., { \bf  29}, 85--98,  (2006).


\bibitem{delanghe}  Delanghe, R. \textit{On Hilbert modules with reproducing kernel}, Functional and Theoretical Methods:Partial Differential Equations, Proceedings of the International Symposium. Darmstadt, 1976, Lecture Notes in Mathematics, { \bf  561},   158--170, (1976).
 
\bibitem{GenSS}
 Gentili, G.,  Stoppato, C.,   Struppa, D. C. \textit{Regular functions of a quaternionic variable},
 Springer Monographs in Mathematics. Springer, Heidelberg,  (2013).

\bibitem{advances}  Gentili, G.,   Struppa,  D.C.  \textit{ A new theory of regular functions
of a quaternionic variable},  Adv. Math., {\bf 216},
279--301,   (2007).

\bibitem{O1}   Gonz\'alez-Cervantes, J. O. \textit{A Fiber Bundle over the Quaternionic Slice Regular Functions}. Adv. Appl. Clifford Algebras 31, 55, (2021). DOI: 10.1007/s00006-021-01158-z

\bibitem{O2}  Gonz\'alez-Cervantes, J. O.   \textit{Quaternionic slice regular functions with some sphere bundles}.  Complex Variables and Elliptic Equations, 67:12, 3036--3047,  (2022). DOI: 10.1080/17476933.2021.1971658


\bibitem{O3}  Gonz\'alez-Cervantes, J. O.    \textit{On fiber bundles and quaternionic slice regular functions}.   Complex Anal. Oper. Theory 16, 72 (2022). https://doi.org/10.1007/s11785-022-01253-4

\bibitem{O4}  Gonz\'alez-Cervantes, J. O.   \textit{An  extension of slice regular functions in terms of fiber bundle theory}. 
Adv.  Appl. Clifford Algebras,  submitted, AACA-D-23-00064R1,   (2023).
 
\bibitem{HJ} Heidrich, R.,   Jank, G.  \textit{On iteration of quaternionic M\"obius transformation},
Compl. Var. Theory Appls.,    {\bf 29},  313--318,  (1996).



\bibitem{NS}  Steenrod, N. \textit{The topology of fibre bundles}, Princeton University Press, Princeton NJ, (1951).
 
\bibitem{shavas}   Shapiro, M.,  Vasilevski, N. \textit{On the Bergman kernel function in hyperholomorphic analysis},
Acta Appl. Math., {  46} 1977, 1--27.

 \bibitem{shavas2}  Shapiro, M.,  Vasilevski,  N. \textit{On the Bergman kernel function in Clifford analysis}, In
 Clifford Analysis and Their Applications in Mathematical Physics, Bracks F. et al (eds.). Proceedings of the Third Conference, Deinze,
 Belgium, 1993. Dordrecht: Kluver Academic Publisher. Fundamental Theories of Physics, { \bf  55}, 183--192, (1993).

 \bibitem{shavas3}  Shapiro, M.,  Vasilevski, N. \textit{On the Bergman kernel functions in quaternionic ana\-ly\-sis},
 Russian Mathematics -  Izv.  VUZov, { \bf 42}, $\sharp$ 2, 81--85, (1998).


\bibitem{W}  Weatherall,   J.O.  \textit{Fiber bundles, Yang-Mills theory, and general relativity}, Synthese, { \bf 193},  2389--2425,  (2016).




















\end{thebibliography}
%

\end{document}